\documentclass[12pt, reqno]{amsart}
\usepackage{amsmath,latexsym,amssymb,amsfonts,amsbsy, amsthm}
\allowdisplaybreaks[4]
\usepackage{graphicx}
\usepackage{empheq}
\usepackage{color}
\usepackage{ulem,cite}
\usepackage{amsthm}
\usepackage[margin=1in]{geometry}

\numberwithin{equation}{section}
\newtheorem{Proposition}{Proposition}
\numberwithin{Proposition}{section}
\newtheorem{Theorem}{Theorem}
\numberwithin{Theorem}{section}
\newtheorem{Lemma}{Lemma}
\numberwithin{Lemma}{section}
\newtheorem{Remark}{Remark}
\numberwithin{Remark}{section}

\numberwithin{Corollary}{section}

\newcommand{\p}{\partial}

\begin{document}
\title[Baroclinic instability for Boussinesq equations]{On the baroclinic instability of  inviscid non-conducting Boussinesq equations with rotation  in 3-D}

\author{Jingjing Mao and  Yan-Lin Wang} 

\address[J. Mao]{School of Science, Zhejiang University of Science and Technology, Hangzhou 310023, China}
\email{{mao.jingjing@outlook.com}}

\address[Y.-L. Wang]{School of Mathematical Sciences, Zhejiang University of Technology,
Hangzhou 310023, China.}
\email{ ylwang@zjut.edu.cn; yanlinwang.math@gmail.com}

\begin{abstract}

In this paper,  we prove the  nonlinear instability of  a  given  vertical  shear of velocity between two rigid plane  for the 3-D inviscid,  non-conducting Boussinesq equations with rotation. When the Rossby number is zero,  this rotating inviscid Boussinesq system reduces to  the nonlinear  geostrophic limit  model.   For non-zero small  Rossby numbers,  we establish the nonlinear instability of the shear flow, which is consistent with that of  the  geostrophic limit model.  The proof relies on constructing a precise  approximate solution, which  comprises  a growing profile derived from  the nonlinear geostrophic limit model  and a   higher-order  asymptotic  expansion with respect to the small Rossby number.  Notice that  the instabilities  (growing modes)  are  driven by the physical boundaries.

\vspace{0.5cm}


\end{abstract}

\let\origmaketitle\maketitle
\def\maketitle{
  \begingroup
  \def\uppercasenonmath##1{} 
  \let\MakeUppercase\relax 
  \origmaketitle
  \endgroup
}
\maketitle


\section{Introduction} 

 We consider the three-dimensional  inviscid non-conducting Boussinesq system  with reference rotation, which  is used to describe  the large-scale  dynamical progress,  such as atmospheric flows and oceanic flows.  The governing system   reads
\begin{equation}\label{original}
\begin{cases}
 R o\left(\partial_{t} u +u \partial_{x} u +v \partial_{y} u +R o w \partial_{z} u  \right)  -v =-\partial_{x}p ,\\
 R o\left(\partial_{t} v +u \partial_{x} v +v \partial_{y} v +R o w \partial_{z} v  \right)  +u =-\partial_{y}p ,\\
R o^2 H^2\left(\partial_{t} w   +u  \partial_{x} w +v  \partial_{y} w  +R o w  \partial_z w \right) =-\partial_{z}  p  +\theta ,\\
\partial_{t}\theta +u \partial_{x}\theta +v \partial_{y} \theta +R o w \partial_{z}  \theta+B w =0,\\
\partial_{x} u +\partial_{y} v +R o \partial_{z}  w  =0,
\end{cases}
\end{equation}
where  $(t, x, y, z) \in \mathbb{R}_+\times \mathbb{T}^2\times[0, 1],$  with horizontal periodic domain $\mathbb{T}^2:=(0, 1)^2.$   $(u, v, w)$, $p$ and $\theta$ stand for  velocity, pressure and  temperature of the fluid.
Here  $H$, $Ro$  and $B$ are, respectively,   the  height  scale  of the flow,  Rossby number and  Burger number.   For example, for the westerlies in the atmosphere,  it is found that $R o \ll 1$ and $B \approx 1$.  
For the readers' convenience, we present the derivation of \eqref{original} in the Appendix, in which the definition of $Ro$ and $B$ will be given.  For more  background and its fundamental  importance of the rotating  Boussinesq equations in both geophysics and mathematics, we refer  to  \cite{Drazin,Hart1979, Majda2003,  pedlosky} and references therein.  
  
 The system \eqref{original} is supplemented with  the boundary conditions:
\begin{align}w=0,\ \ \ {\rm at}\ \ \ z=0, 1. \label{orbc}\end{align}

We take a  basic shear flow of  the  system \eqref{original} as follows:
\begin{align}\label{shear}
\mathbf{U}=(z, 0, 0), \quad \Theta=-y, \quad P=-y z,
\end{align}
where $(x,y)\in \mathbb{T}^2$  and $0\leq z\leq 1.$  Note that this setup of  zonal flow confined between two rigid horizontal plane, together with the boundary condition \eqref{orbc},   is due to the famous Eady model  in atmospheric  science  (see, \cite{Eady1949}).  The shear flow \eqref{shear}  is taken from the physical consideration,  which is consistent with the pressure  distribution of  westerlies.  The instability of the shear flow \eqref{shear} is called {\it baroclinic instability} in meteorology,   because  it depends essentially upon the difference between the surfaces of constant density and of constant pressure in the fluid ( see \cite{Drazin} for the physical mechanism, for instance).   

Formally, we take the geostrophic limit  $Ro\rightarrow 0$  in \eqref{original} to have 
\begin{align}
&u =-\partial_{y} p  , \quad v =\partial_{x} p  , \quad \theta =\partial_{z} p ,\label{Ro1}\\
 &\p_t(\p_y u-\p_x v)+u\p_x (\p_y u-\p_x v)+v\p_y (\p_y u-\p_x v)
+\partial_{z} w=0,\label{Ro2}\\
&\p_t \theta+u\p_x \theta+v\p_y \theta  +B w =0,\label{Ro3}
\end{align}
which is called the  {\it geostrophic limit   model} in the rest of this paper. 
Note that  the basic shear flow \eqref{shear} also satisfies the geostrophic limit  model.  Actually, the equations \eqref{Ro1}-\eqref{Ro3} can be rewritten as a nonlinear  equation of $p$:
\begin{align}
(\p_t-\p_y p \p_x+\p_x p\p_y)[(B(\p_{xx}+\p_{yy})+\p_{zz})p]=0. \label{scalerp}
\end{align}
That is, when $Ro\rightarrow 0$, the system \eqref{original}  can be rewritten   as  an   equation of  $p$ formally.

\subsection{Motivation and the problem}
The baroclinic instability theory was date  from the work of Charney \cite{Charney} in 1947.  Since then extensive progresses have been made.  In 1949, Eady   \cite{Eady1949}  first  revealed the physical  mechanism of (linear) baroclinic instability.  Eady showed that the two sets of Rossby waves cor­responding to both boundaries can give rise to exponential instability.

The instability of the shear flow \eqref{shear} in  Boussinesq system 
is an important topic in both atmospheric science and mathematics. 
  When $Ro=0,$  the linear  (in)stability of shear flow \eqref{shear} has been investigated in \cite{Eady1949, Drazin}.  
 However, if $Ro$ is small but nonzero, the system \eqref{original} can not be  written  into  a single equation of $p$. In this situation, the nonlinear instability of the shear flow \eqref{shear} is   interesting and complicated in mathematical analysis. 
 In this paper, our purpose is to study the nonlinear instability of the shear flow \eqref{shear} in the inviscid non-conducting Boussinesq equations with a  small Rossby number $Ro> 0.$ 
 
To approach this  purpose,  we set
 $$\mathbf{u}=\mathbf{U}+\mathbf{u}^{*}, \quad \theta=\Theta+\theta^{*}, \quad p=P+p^{*},$$
where $\mathbf{u}=(u, v, w)$ and $\mathbf{u}^*=(u^*, v^*, w^*)$. Then  we write the system \eqref{original} into a  perturbed  form:
\begin{equation}\label{perturbation_equations}
\begin{cases}
 R o\left(\partial_{t} u^{*} +z \partial_{x} u^{*} +u^{*} \partial_{x} u^{*} +v^{*} \partial_{y} u^{*}+R o w^{*} \partial_{z} u^{*} +R o w^{*}\right)  -v^{*}\\
 =-\partial_{x} p^{*} ,\\
R o\left(\partial_{t} v^{*} +z \partial_{x} v^{*}+u^{*} \partial_{x}v^{*} +v^{*}\partial_{y} v^{*} +R o w^{*} \partial_{z}v^{*} \right)+u^{*} =-\partial_{y} p^{*}, \\
R o^2 H^2 \left(\partial_{t} w^{*} +z \partial_{x} w^{*} +u^{*} \partial_{x} w^{*} +v^{*} \partial_{y}w^{*} +R o w^{*} \partial_{z} w^{*} \right)\\
 =-\partial_{z} p^{*} +\theta^{*},\\
\partial_{t} \theta^{*} +z \partial_{x} \theta^{*} +u^{*} \partial_{x} \theta^{*} +v^{*} \partial_{y} \theta^{*} +R o w^{*} \partial_{z} \theta^{*} -v^{*}+B w^{*} =0,\\
\partial_{x} u^{*} +\partial_{y} v^{*} +R o \partial_{z} w^{*}  =0,
\end{cases}
\end{equation}
which is complemented with the  following boundary conditions 
\begin{align}\label{boundary_conditions_w}
w^*=0 \quad \text { at } z=0,1 .
\end{align}

Suppose that  a small perturbation is given  to the shear flow initially.  If  there exists a time such that this perturbation grows larger than the scale  of initial perturbation and away from the shear flow, we say the shear flow is unstable.  In this paper we will give a rigorous verification  on the nonlinear instability of the shear flow in the inviscid non-conducting Boussinesq equations with rotation. We will first verify the existence of a growing mode in the nonlinear geostrophic limit model.  Then we construct an approximate solution  for the case $Ro>0$  containing  the shear flow and  a growing profile deduced from the geostrophic limit model.     Then the nonlinear instability of  the shear flow in  inviscid non-conducting Boussinesq  system  can be established through an asymptotic estimates in Sobolev space.   

\subsection{Previous works and main results}
The local existence and blow-up criteria   of strong solution of system \eqref{original} with $Ro=1$ and without rotation  have been established in \cite{ChaeN} in Sobolev space $H^3$. The  inviscid non-conducting Boussinesq equations are constructed  by incompressible Euler equations coupling with  an advective scalar equation. Hence  the Boussinesq system have  some common features of Euler equations \cite{BKM, kato1972, ponce_1985, TR}. 
For the last two equations in \eqref{original} in 2-D, together with  the additional condition $\textbf{u}=(-\p_y (-\Delta)^{-1/2}\theta, \p_x (-\Delta)^{-1/2}\theta)$,  they form the famous inviscid surface  quasi-geostrophic (SQG) equation. One can refer to some recent works on 2-D SQG equation  \cite{FLR} for the spin of an almost-sharp front,  \cite{KRYZ} for the  singularity formation in  finite time , 
\cite{BSV} for nonuniqueness of weak solutions, \cite{CIN} for the zero-viscosity  limit  and  \cite{CGI, CCG} for the  global existence of smooth solutions. 

In view of the well-posedness theory for  Boussinesq system with  full or partial  dissipation,  there are extensive progress.  The  global existence theory  of 2-D Boussinesq equations with either  zero diffusivity or  zero viscosity has been established by  Chae \cite{chae2006}.  Cao and Wu \cite{Caowu}  proved the existence of global  solution in $H^2(\mathbb{R}^2)$ for anisotropic Boussinesq system with only vertical dissipation. Then  the regularity  on  the initial data for the global existence theory  has been relaxed  in the work \cite{Lititi}.  More recently, Chen and Hou \cite{CH}  proved  the finite time singularity formation of $C^{1,\alpha}$ solution for the 2-D Boussinesq equation.   For more well-posedness theory of the  Boussinesq system with  anisotropic dissipation, 
we refer to \cite{AH, DP1, DP2, houli, KM,  LPZ} and references therein.

Recently, the stability  of  shear flow in  Boussinesq system has attracted a lot of attention.  
As for the three-dimensional continuously stratified rotating Boussinesq model with viscosity and diffusivity,  Seng\"{u}l  and  Wang  studied the nonlinear stability and dynamic transitions of the basic  shear flow \eqref{shear} in \cite{SW},   in which numerical analysis has also been given. 
Yang and Lin  \cite{LY} considered the linear stability of  Couette flow $(y, 0)^\top$ and an exponentially stratified density  $\Theta=A e^{-b y}$ ($A$ and $b$ are constants) and obtained the decay rates.  Then Bianchini, Coti Zelati and Dolce \cite{BCD} extended the linear stability results in \cite{LY}  to the case of  shear near Couette flow.  Moreover,  Bedrossian, Bianchini, Coti Zelati and Dolce \cite{BBCD} verified  the nonlinear shear-buoyancy
instability of $U=(y, 0)$ and $\rho=\bar \rho-by$ ($\bar \rho$ and $b$ are constants) up to the timescale $t\approx \varepsilon^{-2}$ if the initial perturbations are of size $\varepsilon$ in Gevrey-$s$, $1\leq s<2.$  In their work \cite{BBCD}, the Rossby number is assumed $Ro=1.$  More recently, Bianchini, Coti Zelati and  Ertzbischoff verified the instability of the  stratified shear flows $(\rho, u, v, p)(x, z)=(\rho(z), U(z), 0, p(z))$ with $p'(z)=-\rho(z)$ in  2-D  Euler-Boussinesq system when the classical Miles-Howard criterion is violated,  and proved the ill-posedness of the limiting hydrostatic equations near the stratified shear flows in Sobolev space \cite{BCZE}.     For a shear flow $U=(y, 0)$ and a stratified temperature $\Theta= \alpha y$ with $\alpha>1/4,$ Zillinger  considered its (in)stability problem in two dimensional  inviscid  Boussinesq equations  \cite{Zi3},  in which a non-trivial  growing solution of the nonlinear inviscid Boussinesq equations has been presented  and  the improved upper bound on norm inflation has been achieved. While  in this paper we will consider the baroclinic instability of shear flow $U=(z, 0, 0)$, $\Theta=-y$ and $P=-yz$ (which is different from the shear flow in the previous research) in   3-D inviscid  non-conducting Boussinesq system  with the reference rotation and small Rossby number in Sobolev space.  For more results on  (in)stability of shear flow in Boussinesq system with full or partial dissipation, we refer to \cite{Fri, ZZhao, ZZ, Zi1, Zi2} and references therein.

 For  nonlinear instability of the  background flow in inviscid fluid  system other than Boussinesq system,  we  mention  some  remarkable progress as follows.  Guo, Hallstrom and  Spirn
\cite{guo_2007} proved the nonlinear instability for vortex sheets
with surface tension, Hele-Shaw flows with surface tension, and vortex patches,  by means of Guo-Strauss's approach arose in  \cite{GS}. The nonlinear instability of  the given background solution  would  happen up to  the time scale  $\log{\frac{1}{\delta}}$ ($\delta$ is the scale of initial perturbation).  We notice that  the existence of  fastest exponential  growing mode  of the linear perturbed system  is crucial to establish the nonlinear instability in Guo-Strauss's approach.  On the other hand, Grenier \cite{Grenier} introduced another  method by constructing  a more precise approximate solution with a growing profile from the linearized  perturbed  equation.  For more nonlinear instability results of inviscid fluid, we refer to \cite{CS, LW, LZ,  XZ, SSB} and references therein.

\textbf{Guo-Strauss's method and Grenier's method.} More precisely, one denotes the   linearized  system around the  shear flow  as follows
$$\p_t \mathcal{U}+A\mathcal{U}=0,\ \ \ \ \ \ \ \ \ \  \text{(LE)}$$
where $A$ is a linear operator. Guo-Strauss's approach in  \cite{GS, guo_2007} requires  the estimates of generator  $e^{t A}\leq C_A e^{\lambda t}$ for some positive constant $C_A$.  Moreover, the  nonlinear  estimates  rely on the existence of the  largest  exponential growing mode.  On the other hand,  Grenier's approach requires to evaluate the spectral radius of $A,$ and then construct a more precise approximate solution $u_{app}=u_{shear}+\delta u_1e^{\lambda_0 t}+\delta \bar{u_1}e^{\bar\lambda_0 t}+u_R,$  where $u_1e^{\lambda_0 t}, \bar u_1e^{\bar\lambda_0 t}$ solves the linear equation (LE) and $u_R$ is the higher order expansion.  The instability is deduced from the growing profile and  the smallness of the  remainder of the approximate solutions.

\textbf{The method of our proof.}
In this paper we present  a new approach to  derive the nonlinear instability of the shear flow in inviscid non-conducting Boussinesq equations with fast rotation (small Rossby number). We will  not search the growing modes from the linear system (LE) for $Ro\neq 0$. Instead, we verify the existence of a specific growing mode in  the nonlinear geostrophic limit  model (where $Ro=0$). Then we construct an approximate solution built upon the nonlinear geostrophic limit  model.  In our paper,  the growing mode is not  necessary to be the largest one.

To establish the nonlinear baroclinic instability for 
Boussinesq system \eqref{original}  with a small Rossby number,  we start with  the  geostrophic limit  model.  
 In particular,  we suppose   a  perturbation   in the perturbed system \eqref{perturbation_equations} with  the  normal modes:
\begin{equation}\label{p_ee}
p(x, y, z)=\hat{p}(z) \exp(\mathrm{i} \alpha(x-c t))\sin \beta y,\ \ \ 0\leq z\leq 1,
\end{equation} 
where $c\in \mathbb{C},$   and  $\alpha, \beta\in \mathbb{Z}\backslash \{0\}$  denote  the  wavenumbers in $x$ and $y$ direction.  As we know, when $Ro=0,$ the linearized  system of  \eqref{perturbation_equations}  will reduce to an ODE  of  $\hat p$.  Let $\Re(\cdot)$ and  $\Im(\cdot)$ denote the real part and imaginary part of a complex number, respectively. Clearly, if $ \Re (-i \alpha c)<0,$ the shear flow is stable.  
If $ \Re (-i \alpha c)>0 $
 (i.e. $\Im(c) \alpha >0$), there exists an exponentially growing mode such that  the shear flow is unstable with respect to this perturbation.  In fact,   the exponential  growth rate is  deduced from  the  boundary conditions for certain  Burger number and  wavenumbers $\alpha, \beta.$ (cf. \cite{Drazin}).  
Notice  that  we  will  consider the nonlinear instability of the shear flow when $Ro >0.$  In this situation  the  perturbed  system \eqref{perturbation_equations} can not  be treated as a single equation  of $\hat{p}$ again. Therefore, we construct  an approximate  solution to the perturbed  system \eqref{perturbation_equations} with high order expansion in terms of $Ro$. The approximate solution near the shear flow  is formed by   the solution of  nonlinear geostrophic limit  model and the high order expansion in terms of $Ro$.  So it is  crucial  to verify  the existence of  a  growing profile in the nonlinear geostrophic limit model. Then combining the nonlinear instability of the shear flow in  the  nonlinear geostrophic limit model and the reasonable asymptotic expansion of  the solution to the  perturbed  Boussinesq system,  we will  establish our nonlinear instability theory for the case $Ro\neq 0$.
 
 \textbf{Notation.} In the rest parts of this paper, we may drop the superscript $*$  in the perturbed equations  \eqref{perturbation_equations}  \eqref{boundary_conditions_w} if it wouldn't   cause confusion.  
We will use $C$ to denote a generic positive  constant.  We denote 
$$\int:=\int_{\mathbb{R}^2\times [0,1]},\ \ {\rm or}\ \ \int:=\int_{\mathbb{T}^2\times [0,1]},$$
and $\|\cdot\|$ for the standard $L^2$-norm defined over ${\mathbb{R}^2\times [0,1]}$ or ${\mathbb{T}^2\times [0,1]},$ and $\|\cdot\|_{\infty}$ for the $L^\infty$ norm. We use $\Im (\cdot)$ to denote imaginary part and $\Re(\cdot)$ to denote the real part. 

Let $$q=\frac{1}{2}\sqrt{B(\alpha^2+\beta^2)},$$
where $B>0$ is Burger number, and  $\alpha, \beta>0$ are for  wavenumbers  defined in \eqref{p_ee}. Now we are ready to state our nonlinear instability result as follows:
 
\begin{Theorem}\label{mainth}
Assume $Ro>0$ and $0<\delta<1.$
Suppose that  $(\textbf{u}^*, \theta^*)(\cdot, t)$ is  a strong solution to the perturbed system \eqref{perturbation_equations} \eqref{boundary_conditions_w} satisfying 
$$\|(\mathbf{u}^*,\theta^*)(\cdot, 0)\|_{L^2}\leq \delta.$$
 There exists a positive  constant $q_c$ such that 
if $0<q<q_c,$ then there is  a time $T^\delta=O(\log \frac{\eta+C Ro}{\delta})$ so that 
\begin{align*}
\sup _{0 \leq t \leq T^\delta}\|(\mathbf{u}^*(\cdot, t),\theta^*(\cdot, t))\|_{L^2} \geq \eta 
\end{align*}
for some positive
 constant $C$  and  $1\geq \eta>\delta>0.$ \end{Theorem}
 
 \textbf{Remark.} We do not have the $L^\infty$-form estimate in the present result, due to the multi-scale singular limit  from \eqref{original} to \eqref{scalerp} in $L^\infty$-form is unsolved at present.  More precisely, the singularities  can be observed from the following linear system of \eqref{original}:
\begin{equation*}
\begin{cases}
& \p_t u=\frac{1}{Ro}(v-\p_x p),\\
 &\p_t v=\frac{1}{Ro}(-u-\p_yp),\\
 &H^2\p_t w=\frac{1}{Ro^2}(\theta-\p_z p)\\
 &\p_t \theta+Bw=0,\\
 &\p_x u+\p_y v+Ro \p_z w=0,
 \end{cases}
 \end{equation*}
 which exhibits multi-scale singularity in $Ro.$ We will left this multi-scale singular limit problem in our future research.   Meanwhile,  for the single scale singular limit (the order of singularity in $Ro$ is same, cf.\cite{BB94, Charve, BLT} for instance),  if  a "well-prepared" initial data is given, Beale  and Bourgeois proved the validity of the geostrophic limit  when Burger number  is  a fixed constant,  using an asymptotic expansion \cite{BB94} (see also  \cite{Charve}).  Recently, another rigorous  mathematical derivation of the single scale singular limit  from the inviscid  Boussinesq equations with rotation has been established  by Bardos, Liu and Titi \cite{BLT} for some general  initial data, using fast waves correction and compactness argument.   On the other hand,  Bresch, G\'erard-Varet and Grenier \cite{BGVG} studied the geostrophic limit for the case of  small Burger number, which is in the same order of Rossby number. For more relevant progress on the geostrophic limit, we refer to \cite{CDGG} and references therein.

The arrangement of the remaining parts  is as follows. Section 2 is  devoted 
to explaining the existence of growing modes in  both linear and nonlinear  geostrophic limit  model. The construction of approximate solution and verification of nonlinear baroclinic instability will be given in Section 3.  In the end of Section 3, we will verify our main result on the nonlinear baroclinic instability.

\section{Instability for the case $Ro=0$}\label{sec2}
The linear instability of shear motion \eqref{shear} has been verified by \cite{Eady1949} in a simplified  dynamic equation.  One can also find  a simplified survey for the linear  geostrophic limit  model  in \cite{Drazin}. 
For the completeness, we present  the main idea of deriving the growing modes in  the linear geostrophic limit model as follows.  
\subsection{Linear instability for $Ro=0$}
We begin with the linearized   equations  of (\ref{perturbation_equations}):
\begin{equation} \label{linear}
\begin{cases}
R o\left(\partial_{t} u +z \partial_{x}u+Row\right)-v =-\partial_{x} p,\\
R o\left(\partial_{t} v+z \partial_{x} v\right)+u =-\partial_{y} p, \\
R o^2 H^2 \left(\partial_{t} w+z \partial_{x} w\right)=-\partial_{z}p+\theta,\\
\partial_{t} \theta+z \partial_{x} \theta-v+B w =0,\\
\partial_{x} u+\partial_{y} v+R o \partial_{z} w =0.
\end{cases}
\end{equation}
Elimination of $p $ from equations $(\ref{linear})_{1}$ and $(\ref{linear})_{2}$ gives the vorticity equation,
$$
R o(\partial_{t}+z \partial_{x} )(\partial_{x} v  -\partial_{y} u )-R o^2 \partial_{y} w  +\partial_{x} u  +\partial_{y}v  =0.
$$
This, together with $\eqref{linear}_5$,  implies that
\begin{align}\label{linear_u_v}
\begin{split}
(\partial_{t}+z \partial_{x} )(\partial_{x} v  -\partial_{y} u) -R o \partial_{y} w  -\partial_{z}w  =0,
\end{split}
\end{align}

Let $R o = 0$.  It follows from \eqref{linear} and \eqref{linear_u_v}  that
\begin{equation}\label{puv}
u =-\partial_{y} p, \quad v =\partial_{x} p, \quad \theta =\partial_{z} p, 
\end{equation}
and 
\begin{equation}\label{Ro-0}
\begin{split}
&\partial_{t} \theta+z \partial_{x} \theta-v+B w =0,\\
&(\partial_{t}+z \partial_{x} )(\partial_{x} v  -\partial_{y} u ) -\partial_{z}w  =0.
\end{split}
\end{equation}
With the aid of \eqref{puv},  we  eliminate $w$ and   write \eqref{Ro-0} into the following form
\begin{align}\label{linear_p}
(\partial_{t}+z \partial_{x} )\big(B(\partial_{xx} p  +\partial_{yy} p )+\partial_{zz} p \big)=0 .
\end{align}
The boundary condition (\ref{boundary_conditions_w}) gives
\begin{equation}\label{boundary_conditions_w_linear_R_0}
(\partial_{t}+z \partial_{x} ) \partial_{z} p  -\partial_{x} p  =0 \quad \text { at } z=0,1 .
\end{equation}

Take the normal modes with
\begin{equation*}\label{p_e}
p =\hat{p}(z) \exp (\mathrm{i} \alpha(x-c t)) \sin \beta y ,
\end{equation*} 
where $c\in \mathbb{C}$ and  $\alpha, \beta \in \mathbb{Z} \backslash \{0\}$.
Equation $(\ref{linear_p})$ gives
\begin{align*} 
(z-c)\left(\partial_{zz} \hat{p} -B(\alpha^2+\beta^2) \hat{p}\right)=0 .
\end{align*}
From here, we can not  directly obtain  the growing modes (Namely, $\Re (c)>0$). Instead, for any $z,$ we take 
\begin{align*} 
\partial_{zz} \hat{p} -B(\alpha^2+\beta^2) \hat{p} =0.
\end{align*}
Hence one  can solve 
\begin{align}
\hat{p}=R \cosh 2 q z+S \sinh 2 q z, \label{solphat}
\end{align}
where
$$
q=\frac{1}{2}(B(\alpha^2+\beta^2))^{1 / 2},
$$
and $R, S$ are some constants.     From the  boundary conditions \eqref{boundary_conditions_w_linear_R_0},  we have
\begin{align*}
&c\p_z \hat{p}(0)+\hat p(0)=0\ \ \ {\rm at} \ \ z=0,\\
&c \p_z \hat{p}(1)-\p_z \hat{p}(1)+\hat{p}(1)=0 \ \ \ {\rm at} \ \ z=1.
\end{align*}
Submitting \eqref{solphat} into the above boundary conditions,   one can eliminate $R$ and $S$ to solve
\begin{align*}\label{cformula}
c=\frac{1}{2}\pm(q^2+1-2q\coth (2q))^{1/2}/(2 q).
\end{align*} 
Hence it holds  $\alpha \Im(c) >0$ if and only if $q<q_c,$ where $q_c (\approx 1.2)$ is the positive root of $q^2+1-2q\coth (2q)=0.$ That is, it has  an exponential growing mode $e^{\alpha \Im(c) t}$ when $B(\alpha^2+\beta^2)<4q_c^2.$ Of course, the largest growing rate $\lambda_{\max}:=\max\{\alpha \Im(c)\}$ can be obtained for a   $q\in (0, q_c).$


Next, we will present the nonlinear instability  of the shear flow for the case $R_0=0$.
\subsection{Nonlinear instability for $Ro=0$}
We recall the perturbation equations  (\ref{perturbation_equations}) that
\begin{equation}
\begin{cases}\label{22equ}
 R o\left(\partial_{t} u  +z \partial_{x} u  +u  \partial_{x} u  +v  \partial_{y} u +R o w  \partial_{z} u  +R o w \right)  -v  =-\partial_{x} p  ,\\
R o\left(\partial_{t} v  +z \partial_{x} v +u  \partial_{x}v  +v \partial_{y} v  +R o w  \partial_{z}v  \right)+u  =-\partial_{y} p , \\
R o^2 H^2 \left(\partial_{t} w  +z \partial_{x} w  +u  \partial_{x} w  +v  \partial_{y}w  +R o w  \partial_{z} w  \right) =-\partial_{z} p  +\theta ,\\
\partial_{t} \theta  +z \partial_{x} \theta  +u  \partial_{x} \theta  +v  \partial_{y} \theta  +R o w  \partial_{z} \theta  -v +B w  =0,\\
\partial_{x} u  +\partial_{y} v  +R o \partial_{z} w =0.
\end{cases}
\end{equation}
Elimination of $p $ from equations $(\ref{22equ})_{1}$ and $(\ref{22equ})_{2}$ gives the vorticity equation,
\begin{align*}
&R o(\partial_{t}+z \partial_{x}+u \partial_{x} +v \partial_{y})(\partial_{x} v  -\partial_{y} u )-R o^2 \partial_{y} w    \\
&\quad+R o(\partial_{x} u +\partial_{y} v ) (\partial_{x} v  -\partial_{y} u  )
+\partial_{x} u  +\partial_{y}v\\
&\quad+R o^2\left(\partial_{x} w  \partial_{z} v  -\partial_{y} w  \partial_{z} u 
+w \partial_{z}(\partial_{x} v  -\partial_{y} u  ) \right)=0,
\end{align*}
and then
\begin{align}\label{nonlinear_u_v}
\begin{split}
&(\partial_{t}+z \partial_{x}+u \partial_{x} +v \partial_{y})(\partial_{x} v  -\partial_{y} u )-R o\partial_{y} w     \\
&\quad+(\partial_{x} u +\partial_{y} v ) (\partial_{x} v  -\partial_{y} u  )-\partial_{z} w\\
&\quad+R o\left(\partial_{x} w  \partial_{z} v  -\partial_{y} w  \partial_{z} u  
+w \partial_{z}(\partial_{x} v  -\partial_{y} u  ) \right)=0,
\end{split}
\end{align}
where we used  $\eqref{22equ}_{4}$. In the geostrophic limit as $R o \rightarrow 0$, equations  $\eqref{22equ}_{1}$, $\eqref{22equ}_{2}$ and $\eqref{22equ}_{3}$ become
\begin{equation}\label{nonlinear_u_v_theta}
u =-\partial_{y} p  , \quad v =\partial_{x} p  , \quad \theta =\partial_{z} p  
\end{equation}
From $\eqref{22equ}_{4}$, (\ref{nonlinear_u_v}) and \eqref{nonlinear_u_v_theta}, it is clear that 
\begin{align}\label{nonlinear_u_v_p}
(\partial_{t}+z \partial_{x} -\partial_{y} p  \partial_{x} +\partial_{x} p  \partial_{y})
(\partial_{xx} p  +\partial_{yy} p  ) 
=\partial_{z} w   ,
\end{align} 
\begin{align}\label{nonlinear_theta_p}
(\partial_{t}+z \partial_{x} -\partial_{y} p  \partial_{x} +\partial_{x} p  \partial_{y})
\partial_{z} p 
-\partial_{x} p  +B w =0 .
\end{align}
Elimination of $w $ from equations (\ref{nonlinear_u_v_p}) and (\ref{nonlinear_theta_p}) finally gives
\begin{align}\label{nonlinear_p}
(\partial_{t}+z \partial_{x}- \partial_{y} p\partial_{x}   +\partial_{x}  p\partial_{y} )
\big(B(\partial_{xx} p  +\partial_{yy} p  )+\partial_{zz}p\big)=0 .
\end{align}
The boundary condition  (\ref{boundary_conditions_w})  gives
\begin{align} \label{nonlinear_p_boundary_2}
(\partial_{t}+z \partial_{x} )
 \partial_{z}  p  
-\partial_{y} p  \partial_{xz} p  
+\partial_{x}  p  \partial_{yz} p  
-\partial_{x}  p  =0  \text { at } z=0,1 .
\end{align}

\begin{Lemma}
The solution  of  the linear equations \eqref{linear_p} \eqref{boundary_conditions_w_linear_R_0}  satisfying  the form \eqref{p_e}  and  $\left(\partial_{zz} \hat{p} -B(\alpha^2+\beta^2) \hat{p}\right)=0$ 
also solves the nonlinear equations \eqref{nonlinear_p} \eqref{nonlinear_p_boundary_2}.  Moreover, the nonlinear perturbed equation \eqref{nonlinear_p}    \eqref{nonlinear_p_boundary_2} also has an exponential growing mode $e^{\alpha \Im (c) t}$ when $B(\alpha^2+\beta^2)<4 q_c^2,$  where $q_c$ is the positive root of $q^2+1-2q\coth (2q)=0.$
\end{Lemma}
{\it Proof.}
Indeed, the solution to the linear equations \eqref{linear_p} \eqref{boundary_conditions_w_linear_R_0}  in the form \eqref{p_e}   satisfies the boundary condition \eqref{nonlinear_p_boundary_2}, 
due to 
$-\p_y p\p_{xz}p+\p_x p\p_{yz}p=0.$
So the solution of  the equation  $\partial_{zz} \hat{p} -B(\alpha^2+\beta^2) \hat{p}=0$ with  \eqref{boundary_conditions_w_linear_R_0} also solve  the nonlinear problem \eqref{nonlinear_p} \eqref{nonlinear_p_boundary_2}. One can obtain that 
\begin{align}
&(\partial_{t}+z \partial_{x}- \partial_{y} p\partial_{x}   +\partial_{x}  p\partial_{y} )
\left(B(\partial_{xx} p  +\partial_{yy} p  )+\partial_{zz}p\right)\notag\\
&\quad = (-\partial_{y} p\partial_{x}   +\partial_{x}  p\partial_{y} )
\left(B(\partial_{xx} p  +\partial_{yy} p  )+\partial_{zz}p\right)\notag\\
&\quad=(-\partial_{y} p\partial_{x}   +\partial_{x}  p\partial_{y} )(\partial_{zz}\hat{p}-B(\alpha^2+\beta^2)\hat{p})\sin \beta y \exp{\{\text{i}\alpha (x-ct)\}}\notag\\
&\quad=0.
\end{align}

Now we  conclude that the nonlinear equations \eqref{nonlinear_p} \eqref{nonlinear_p_boundary_2} can   generate   the same growing mode  as the linearized equations  \eqref{linear_p}\eqref{boundary_conditions_w_linear_R_0} 
with respective to the prescribed  perturbation form \eqref{p_e}.   $\square$

\section{Baroclinic instability for the case $Ro\neq 0$}
In this section, we present the nonlinear instability of the shear motion in  the inviscid non-conducting Boussinesq system. The crucial step is  the construction of approximate solution with a growing profile applied to the nonlinear geostrophic limit model. The 
instability will happen up to a  time scale $T^\delta= O(\log \frac{1+C Ro}{\delta}),$ where $C$ is a positive constant and
 $\delta$ is the scale of initial perturbation.
\subsection{Construction of approximate solution.}
Recall that the dimensionless basic flow reads
\begin{equation*} 
\mathbf{U}=(z, 0, 0), \quad \Theta=-y, \quad P=-y z \quad \text { for } \  (x, y) \in \mathbb{T}^2, \quad 0 \leqslant z \leqslant 1.
\end{equation*}
We write the  total flow for the nonlinear  system \eqref{original} as
\begin{align} \label{app}
\mathbf{u}=\mathbf{U}+ (\mathbf{u}_{N}+\mathbf{u}_{e}), \quad \theta=\Theta+ (\theta_{N}+\theta_{e}), \quad p=P+ (p_{N}+p_{e}),
\end{align}
where $\mathbf{u}_N, \theta_N, p_N$ solve \eqref{nonlinear_u_v_theta}, \eqref{nonlinear_u_v_p} and \eqref{nonlinear_theta_p}. That is, 
\begin{align}
&u_{N} =-\partial_{y} p_{N}  , \quad v_{N} =\partial_{x} p_{N}  , \quad \theta_{N} =\partial_{z} p_{N}\label{utpN1},\\
&(\partial_{t}+z \partial_{x} -\partial_{y} p_{N}  \partial_{x} +\partial_{x} p_{N}  \partial_{y})
(\partial_{xx} p_{N}  +\partial_{yy} p_{N}  ) 
=\partial_{z} w_{N}  ,\label{utpN2}\\
&(\partial_{t}+z \partial_{x} -\partial_{y} p_{N}  \partial_{x} +\partial_{x} p_{N}  \partial_{y})\partial_{z} p _{N}
-\partial_{x} p_{N}  +B w_{N} =0 ,\label{utpN3}
\end{align} 
which satisfies the boundary conditions 
 $$w_N=0,\ \ \ {\rm at}\ \ \ z=0, 1,$$  which is equivalent to that
$$(\partial_{t}+z \partial_{x} ) \partial_{z} p_N  -\partial_{x} p_N -\p_yp_N\p_x\p_z p_N+\p_x p_N \p_y\p_z p_N=0 \quad \text { at } z=0, 1.$$

In this paper, we construct an approximation of $(\textbf{u}_e, \theta_e, p_e)$ in the following form:
\begin{align} \label{eap}
\begin{split}
\mathbf{u}_{e}& =R o\mathbf{u}_{e}^{(1)}+R o^2\mathbf{u}_{e}^{(2)}
                 +R o^3\mathbf{u}_{e}^{(3)} +\cdots,\\
\theta_{e}& =R o\theta_{e}^{(1)}+R o^2\theta_{e}^{(2)}
                 +R o^3\theta_{e}^{(3)} +\cdots,\\
p_e& =R op_{e}^{(1)}+R o^2p_{e}^{(2)}+R o^3p_{e}^{(3)} +\cdots,
\end{split}
\end{align}
 where $Ro$ is a small fixed constant but $Ro\neq 0$.
 
 \begin{Lemma}\label{lem4.1} Assume that  $\mathbf{u}$, $ \theta$ and $p$ satisfy the system (\ref{original}) \eqref{orbc}.
Let  \begin{align*} 
\mathbf{u}^{\prime} &=\mathbf{u}-\mathbf{U}-\mathbf{u}_a,\quad
\theta^{\prime}     =\theta-\Theta-\theta_{a},\quad
  p^{\prime}        =p-P-p_{a} ,
\end{align*}
where 
\begin{align*} 
u_{a}&=u_{N}+R ou_{e}^{(1)}+R o^2u_{e}^{(2)}+R o^3u_{e}^{(3)},\\
v_{a}&=v_{N}+R ov_{e}^{(1)}+R o^2v_{e}^{(2)}+R o^3v_{e}^{(3)},\\
w_{a}&=w_{N}+R ow_{e}^{(1)}+R o^2w_{e}^{(2)},\\
\theta_{a}    &=\theta_{N}+Ro\theta_{e}^{(1)}+Ro^2\theta_{e}^{(2)}+R o^3\theta_{e}^{(3)},\\
p_{a}         &=p_{N}+R op_{e}^{(1)}+R o^2p_{e}^{(2)}+R o^3p_{e}^{(3)} .
\end{align*}
There exist $((u_N, v_N, w_N), \theta_N, p_N)((x, y, z), t)\in H^3(\mathbb{T}^2\times [0,1], (0, \infty))$\\
 and 
$((u_e^{(i)}, v_e^{(i)}, w_e^{(i)}), \theta_e^{(i)}, p_e^{(i)})(\cdot, t)\in H^3(\mathbb{T}^2\times [0,1], [0, T_a]), 1\leq i\leq 3,$ such that 
\begin{equation}\label{u_prime}
\begin{split}
      &\partial_{t}u ^{\prime}+z\partial_{x}u^{\prime}
+u^{\prime}\partial_{x} (u ^{\prime}+u_{a})+u_{a}\partial_{x}u ^{\prime}\\
      &\quad
+v^{\prime}\partial_{y}(u^{\prime}+u_{a})+v_{a}\partial_{y}u^{\prime}+Row^{\prime}\partial_{z}(u^{\prime}+u_{a})\\
      &\quad
+Row_{a}\partial_{z}u^{\prime}
+Row^{\prime}=\frac{1}{Ro}(v^{\prime}-\partial_{x}p^{\prime})+Ro^3M_{1},
\end{split}
\end{equation} 
\begin{equation}\label{v_prime}
\begin{split}
&\partial_{t}v^{\prime}+z\partial_{x}v^{\prime}
+u^{\prime}\partial_{x} (v^{\prime}+v_{a})+u_{a}\partial_{x}v^{\prime}\\
      &\quad
+v^{\prime}\partial_{y}(v^{\prime}+v_{a})+v_{a}\partial_{y}v^{\prime}+Row^{\prime}\partial_{z}(v^{\prime}+v_{a})\\
      &\quad
+Row_{a}\partial_{z}v^{\prime}
=\frac{1}{Ro}(-u^{\prime}-\partial_{y}p^{\prime})+Ro^3M_{2},
\end{split}
\end{equation} 
\begin{equation}\label{w_prime}
\begin{split}
      &H^2 Ro \big( \partial_{t}w^{\prime}+z\partial_{x}w^{\prime}
+u^{\prime}\partial_{x} (w^{\prime}+w_{a})+u_{a}\partial_{x}w^{\prime}\\
      &\quad
+v^{\prime}\partial_{y}(w^{\prime}+w_{a})+v_{a}\partial_{y}w^{\prime}+Row^{\prime}\partial_{z}(w^{\prime}+w_{a})\\
      &\quad
+Row_{a}\partial_{z}w^{\prime}\big)
=\frac{1}{Ro}(\theta^{\prime}-\partial_{z}  p^{\prime})  +Ro^3 M_{3} ,
\end{split}
\end{equation} 
\begin{equation}\label{theta_prime}
\begin{split}
      &\partial_{t}\theta^{\prime}+z\partial_{x}\theta^{\prime}
+u^{\prime}\partial_{x}(\theta^{\prime}+\theta_{a})+u_{a}\partial_{x}\theta^{\prime}\\
      &\quad
+v^{\prime}\partial_{y}(\theta^{\prime}+\theta_{a})+v_{a}\partial_{y}\theta^{\prime}
+Row^{\prime}\partial_{z}(\theta^{\prime}+\theta_{a})\\
      &\quad
+Row_{a}\partial_{z}\theta^{\prime} 
-v^{\prime}+Bw^{\prime}=Ro^3M_{4},
\end{split}
\end{equation}  
\begin{equation}\label{div}
\begin{split} 
      &\partial_{x}u^{\prime}+\partial_{y}v^{\prime}+Ro\partial_{z}w^{\prime}=0,
\end{split}
\end{equation}
where $M_j, 1 \leq j\leq 4$ only depends on $((u_e^{(i)}, v_e^{(i)}, w_e^{(i-1)}), \theta_e^{(i)}, p_e^{(i)}), 1\leq i\leq 3.$ Here we set $w_e^{(0)}=0.$ In addition, the error equations \eqref{u_prime}-\eqref{div} satisfy the boundary conditions
$$w'=0,\ \ \ w_a=0,\ \ \ {\rm at}\ \ z=0, 1.$$
 \end{Lemma}
 
 \begin{Remark}
 The expansion order $O(Ro^2)$ is enough in the following  proof  (see the proof in Section \ref{subsec3.2} ). Here we just give an illustration  to show that the expansion can be extended order by order. 
 \end{Remark}
 {\it Proof.}
Substituting the approximate form \eqref{app} into the nonlinear system \eqref{original}, we obtain
\begin{align} 
      & \partial_{t} (u _{N} +u_{e}) 
+ z\partial_{x} (u _{N}+u_{e})
+(u _{N}+u_{e}) \partial_{x} (u _{N}+u_{e}) 
\notag\\  
      &\quad+(v _{N} +v_{e})\partial_{y} (u _{N}+u_{e}) 
+Ro(w _{N}+w_{e}) \partial_{z} (u _{N}+u_{e})\notag\\
      &\quad+Ro(w _{N}+w_{e}) -\frac{1}{Ro} (v _{N}+v_{e})=- \frac{1}{Ro}\partial_{x} (p_{N}+p_{e}) ,\label{line1}\\
     &\partial_{t} (v _{N}+v_{e}) 
+z\partial_{x} (v _{N}+v_{e})
+   (u _{N}+u_{e}) \partial_{x} (v _{N}+v_{e})\notag\\  
      &\quad+  (v  _{N}+v_{e})\partial_{y} (v _{N}+v_{e})
+Ro(w _{N}+w_{e}) \partial_{z} (v _{N}+v_{e})\notag\\ 
      &\quad+ \frac{1}{Ro}(u _{N}+u_{e})
 =- \frac{1}{Ro}\partial_{y} (p _{N}+p_{e}) ,\label{line2}\\
 &H^2 \Big( \partial_{t} (w _{N}+w_{e})  
+ z\partial_{x} (w _{N}+w_{e})
+   (u _{N}+u_{e}) \partial_{x} (w _{N}+w_{e})\notag\\
 &\quad +  (v _{N}+v_{e})\partial_{y} (w _{N}+w_{e}) 
+Ro(w _{N}+w_{e}) \partial_{z} (w _{N}+w_{e})\Big) \notag\\
&\quad=\frac{1}{R o^2}(-\partial_{z} (p _{N}+p_{e}) +(\theta _{N}+\theta_{e})),\label{line3}\\
      & \partial_{t}(\theta_{N} +\theta_{e}) 
+ z\partial_{x}(\theta_{N} +\theta_{e}) 
+ (u _{N}+u_{e}) \partial_{x}(\theta_{N} +\theta_{e})\notag \\
       &\quad+  (v _{N}+v_{e}) \partial_{y} (\theta_{N} +\theta_{e}) 
+Ro (w _{N}+w_{e}) \partial_{z} (\theta_{N} +\theta_{e})\notag\\
      &\quad- (v _{N}+v_{e})+B (w _{N}+w_{e})  =0,\label{line4}\\
      &\partial_{x} (u _{N}+u_{e}) +\partial_{y}(v _{N}+v_{e})+R o \partial_{z} (w _{N}+w_{e}) =0.\label{line5}
\end{align}
It follows from the third equation \eqref{line3} that $$\p_z p_N=\theta_N \ {\rm \  for \  the \ order}\  O(\frac{1}{Ro^2}),$$ and 
$$\p_z p_e^{(1)}=\theta^{(1)}_e\ {\rm  for \ the \ order }\ O(\frac{1}{Ro}). $$ 

Then the order $O(1/Ro)$ for the equations \eqref{line1} and \eqref{line2},   and $O(1)$ order for the vorticity $\p_y\eqref{line1}-\p_x\eqref{line2}$  and $O(1)$ order for  the third equation \eqref{line3},  just form the equations \eqref{utpN1}-\eqref{utpN3}, whose  existence theory   and growing mode have been verified in Section \ref{sec2}.

In fact, using the approximate scheme \eqref{eap}, we have the following equations for the coefficient $O(1):$
\begin{align} 
      &   \underbrace{\partial_{t} u _{N} 
+ z\partial_{x} u  _{N} 
+u _{N}\partial_{x} u  _{N}
+v _{N}\partial_{y} u _{N} }_{L_1}
 =v_{e}^{(1)}- \partial_{x} p_{e}^{(1)} ,\notag\\
      & \underbrace{\partial_{t} v _{N}  
+z\partial_{x} v _{N} 
+   u _{N}  \partial_{x}  v _{N} 
+   v _{N} \partial_{y} v  _{N}    }_{L_2}
=   -u_{e}^{(1)}- \partial_{y} p_{e}^{(1)} ,\notag\\
& H^2 \Big( \partial_{t}  w _{N}  
+ z\partial_{x}  w _{N} 
+  u _{N}  \partial_{x} w _{N} 
+  v _{N}\partial_{y}  w _{N}   \Big) =\theta_{e}^{(2)}-\partial_{z}  p_{e} ^{(2)} ,\label{auv1}\\
      & \partial_{t} \theta_{N} 
+ z\partial_{x} \theta_{N}   
+  u _{N} \partial_{x}\theta_{N}   
+  v _{N} \partial_{y}  \theta_{N}  
-  v _{N}  +B w_{N} =0,\notag\\
      &\partial_{x}  u_N  +\partial_{y} v_N  =0.\notag
\end{align}
Then we collect the
 order  $O(Ro)$   in  $\mathbf{u}_{e}$, $\theta_{e}$ and $p_{e}$ to obtain
\begin{align} \label{u_e_1}
\begin{split}
      &   \partial_{t} u_{e}^{(1)} + z\partial_{x} u_{e}^{(1)}
+ u _{N} \partial_{x} u_{e}^{(1)} +u_{e}^{(1)} \partial_{x} u _{N} 
+ v _{N} \partial_{y} u_{e}^{(1)} \\
&\quad +v_{e}^{(1)}\partial_{y} u _{N}
+w _{N} \partial_{z} u _{N}
+w _{N}
=v_{e}^{(2)}- \partial_{x} p_{e}^{(2)},\\
 & \partial_{t}  v_{e}^{(1)}+z\partial_{x}  v_{e} ^{(1)}
      + u _{N} \partial_{x} v_{e}^{(1)} +u_{e}^{(1)} \partial_{x} v _{N} 
+ v _{N} \partial_{y} v_{e}^{(1)} \\
&\quad+ v_{e}^{(1)}\partial_{y} v _{N} 
+w_{N} \partial_{z} v _{N}  
=-u_{e}^{(2)} - \partial_{y} p_{e}^{(2)} ,\\
& H^2 \Big( \partial_{t} w_{e}^{(1)} + z\partial_{x}w_{e}^{(1)}
+ u _{N} \partial_{x} w_{e}^{(1)} +u_{e}^{(1)} \partial_{x} w _{N} 
+ v _{N} \partial_{y} w_{e}^{(1)} \\
&\quad+v_{e}^{(1)}\partial_{y} w _{N} 
+w _{N} \partial_{z} w_{N}\Big)
=\theta_{e}^{(3)}-\partial_{z} p_{e}^{(3)} ,\\
      & \partial_{t} \theta_{e}^{(1)}
+z\partial_{x} \theta_{e}^{(1)}
+ u _{N} \partial_{x} \theta_{e}^{(1)} +u_{e}^{(1)} \partial_{x} \theta _{N} 
+ v _{N} \partial_{y}\theta_{e}^{(1)}\\
&\quad +v_{e}^{(1)}\partial_{y} \theta  _{N} 
+w _{N} \partial_{z} \theta _{N}
-v_{e}^{(1)}+Bw_{e}^{(1)}=0, \\
      &\partial_{x}  u_{e}^{(1)}  +\partial_{y} v_{e}^{(1)}+\partial_{z}w _{N}  =0.
\end{split}
\end{align}
Similarly, we collect the order $O(Ro^2)$ and $O(Ro^3)$, respectively, to have 
\begin{align*}  
     &\partial_{t} u_{e}^{(2)} + z\partial_{x}  u_{e}^{(2)}
+ u_{e}^{(2)}  \partial_{x} u _{N}+2u_{e}^{(1)}  \partial_{x} u_{e}^{(1)}
+u _{N}  \partial_{x} u_{e}^{(2)} \\
&\quad+ v_{e}^{(2)}  \partial_{y} u _{N}+2v_{e}^{(1)}  \partial_{y} u_{e}^{(1)}
+v _{N}  \partial_{y} u_{e}^{(2)} \\  
     &\quad+ w_{e}^{(1)}  \partial_{z} u _{N}+w _{N}  \partial_{z} u_{e}^{(1)}  +w^{(1)}_{e} =v_{e}^{(3)}-\partial_{x} p_{e}^{(3)} ,\\
     &\partial_{t} v_{e}^{(2)} + z\partial_{x}  v_{e}^{(2)}
+ u_{e}^{(2)}  \partial_{x} v _{N}+2u_{e}^{(1)}  \partial_{x} v_{e}^{(1)}
+u _{N}  \partial_{x} v_{e}^{(2)} \\
&\quad + v_{e}^{(2)}  \partial_{y} v _{N}+2v_{e}^{(1)}  \partial_{y} v_{e}^{(1)}
+v _{N}  \partial_{y} v_{e}^{(2)} \\  
     &\quad+ w_{e}^{(1)}  \partial_{z} v _{N}+w _{N}  \partial_{z} v_{e}^{(1)} =-u_{e}^{(3)}-\partial_{y} p_{e}^{(3)},\\
     & H^2 \Big(\partial_{t} w_{e}^{(2)} + z\partial_{x}  w_{e}^{(2)}
+ u_{e}^{(2)}  \partial_{x} w _{N}+2u_{e}^{(1)}  \partial_{x} w_{e}^{(1)}\\
&\quad +u _{N}  \partial_{x} w_{e}^{(2)} 
+v_{e}^{(2)}  \partial_{y} w _{N}+2v_{e}^{(1)}  \partial_{y} w_{e}^{(1)} \\  
&\quad+v _{N}  \partial_{y}w_{e}^{(2)}+ w_{e}^{(1)}  \partial_{z} w _{N}+w _{N}  \partial_{z} w_{e}^{(1)} \Big)
=\theta_{e}^{(4)}-\partial_{z} p_{e}^{(4)} , 
\end{align*} 
\begin{align*} 
\begin{split}
     & \partial_{t} \theta_{e}^{(2)}+z\partial_{x} \theta_{e}^{(2)}
+ u_{e}^{(2)}  \partial_{x} \theta _{N}+2u_{e}^{(1)}  \partial_{x}\theta_{e}^{(1)}+u _{N}  \partial_{x} \theta_{e}^{(2)} \\
&\quad + v_{e}^{(2)}  \partial_{y} \theta _{N}+2v_{e}^{(1)}  \partial_{y} \theta_{e}^{(1)}+v _{N}  \partial_{y} \theta_{e}^{(2)} \\   
     &\quad +  w _{N} \partial_{z} \theta_{e}^{(1)}
+w_{e}^{(1)} \partial_{z} \theta_{N}  -v_{e}^{(2)}+ B w_{e} ^{(2)}=0,\\
     &\partial_{x} u_{e}^{(2)} +\partial_{y}v_{e}^{(2)}+\partial_{z} w_{e}^{(1)} =0,
\end{split}
\end{align*} 
and
\begin{align*} 
     &\partial_{t} u_{e}^{(3)} + z\partial_{x}  u_{e}^{(3)}
+ u_{e}^{(3)}  \partial_{x} u _{N}+3u_{e}^{(1)}  \partial_{x} u_{e}^{(2)}
+3u_{e}^{(2)}  \partial_{x} u_{e}^{(1)}\\
&\quad+u _{N}  \partial_{x} u_{e}^{(3)} 
+ v_{e}^{(3)}  \partial_{y} u _{N}+3v_{e}^{(1)} \partial_{y} u_{e}^{(2)}\\
     &\quad +3v_{e}^{(2)}  \partial_{y} u_{e}^{(1)}
+v _{N}  \partial_{y} u_{e}^{(3)} + w_{e}^{(2)}  \partial_{z} u _{N}+2w_{e}^{(1)}  \partial_{z} u_{e}^{(1)}\\
&\quad+w _{N}  \partial_{z} u_{e}^{(2)}  +w_{e}^{(2)}=v_{e}^{(4)}-\partial_{x} p_{e}^{(4)} ,\\
     &\partial_{t} v_{e}^{(3)} + z\partial_{x}  v_{e}^{(3)}
+ u_{e}^{(3)}  \partial_{x} v _{N}+3u_{e}^{(1)}  \partial_{x} v_{e}^{(2)}+3u_{e}^{(2)}  \partial_{x} v_{e}^{(1)}\\
&\quad+u _{N}  \partial_{x} v_{e}^{(3)} 
+ v_{e}^{(3)}  \partial_{y} v _{N}+3v_{e}^{(1)} \partial_{y} v_{e}^{(2)}\\
     &\quad+3v_{e}^{(2)}  \partial_{y} v_{e}^{(1)}
+v  _{N}  \partial_{y} v_{e}^{(3)}
+ w_{e}^{(2)}  \partial_{z} v _{N}\\
&\quad
+ 2w_{e}^{(1)}  \partial_{z} v_{e}^{(1)}+w _{N}  \partial_{z} v_{e}^{(2)} =-u_{e}^{(4)}-\partial_{y} p_{e}^{(4)} ,\\
     & \partial_{t} \theta_{e}^{(3)}+z\partial_{x} \theta_{e}^{(3)}
+ u_{e}^{(3)}  \partial_{x} \theta _{N}+3u_{e}^{(1)}  \partial_{x} \theta_{e}^{(2)}+3u_{e}^{(2)}  \partial_{x} \theta_{e}^{(1)}\\
&\quad +u _{N}  \partial_{x} \theta_{e}^{(3)} 
+ v_{e}^{(3)}  \partial_{y} \theta _{N}+3v_{e}^{(1)}\partial_{y} \theta_{e}^{(2)}+3v_{e}^{(2)}  \partial_{y} \theta_{e}^{(1)}
+v _{N}  \partial_{y} \theta_{e}^{(3)}\\
     &\quad
+ w_{e}^{(2)}  \partial_{z} \theta _{N}
+ 2w_{e}^{(1)}  \partial_{z} \theta_{e}^{(1)}+w_{N}  \partial_{z} \theta_{e}^{(2)}  -v_{e}^{(3)}+ B w_{e} ^{(3)}=0,\\
     &\partial_{x} u_{e}^{(3)} +\partial_{y}v_{e}^{(3)}+ \partial_{z} w_{e}^{(2)} =0.
\end{align*}
Combining $\p_z p_e^{(1)}=\theta^{(1)}_{e}$ with $\eqref{auv1}_{1}$ $\eqref{auv1}_{2}$  $\eqref{u_e_1}_4$ and 
$\eqref{u_e_1}_5$, we get
\begin{align} 
      &  \theta_{e}^{(1)}=\partial_{z}  p_{e} ^{(1)},\quad 
v_{e}^{(1)}=L_1+ \partial_{x} p_{e}^{(1)} ,\quad
u_{e}^{(1)}=-L_2-\partial_{y} p_{e}^{(1)},\label{L_1_L_2} \\
      & (\partial_{t}+z \partial_{x}+u_{N}\partial_{x}+v_{N}\partial_{y}) \theta_{e}^{(1)}
 +u_{e}^{(1)} \partial_{x} \theta _{N}+v_{e}^{(1)}\partial_{y} \theta  _{N}
 +w _{N} \partial_{z} \theta _{N}  \notag\\ 
     &\quad  
-v_{e}^{(1)}+Bw^{(1)}_{e}=0, \\
      & \partial_{x}  u_{e}^{(1)}  +\partial_{y} v_{e}^{(1)}+\partial_{z}w _{N}=0.
\end{align}
Next, together with $\eqref{u_e_1}_{5}$,  we calculate plane vorticity $\p_y\eqref{u_e_1}_1-\p_x\eqref{u_e_1}_2$ to obtain 
\begin{align}
    &(\partial_{t}+z \partial_{x}+u_{N}\partial_{x}+v_{N}\partial_{y})(\partial_{x} v_{e}^{(1)}-\partial_{y}u_{e}^{(1)})
+u_{e}^{(1)}\partial_{x}(\partial_{x} v_{N}  -\partial_{y} u_{N}  )\notag\\
   &\quad+v_{e}^{(1)}\partial_{y}(\partial_{x} v_{N}  -\partial_{y} u_{N}  ) +(\partial_{x} u_{e}^{(1)}+\partial_{y} v_{e}^{(1)} )
(\partial_{x} v_{N}  -\partial_{y} u_{N} )\notag\\
   &\quad+\partial_{x}(w_{N}\partial_{z}v_{N})-\partial_{y}(w_{N}\partial_{z}u_{N}+w_{N}) 
=\p_z w_{e}^{(1)},\label{auv-vor}
\end{align}
Finally, we rewrite \eqref{L_1_L_2}-\eqref{auv-vor}, by elimination of $ w_e^{(1)}$,  to get
\begin{align}
    &(\partial_{t}+z \partial_{x}+u_{N}\partial_{x}+v_{N}\partial_{y})\Big(B(\partial_{xx} p_{e}^{(1)}+\partial_{yy}p_{e}^{(1)})+\partial_{zz}p_{e}^{(1)}\Big)\notag\\
   &\quad 
+ \Big(\partial_{y}\big(B(\partial_{x} v_{N}  -\partial_{y} u_{N})+\partial_{z}\theta_{N}\big)+\partial_{y}\theta_{N}\partial_{z}\Big)\partial_{x}p_{e}^{(1)}\notag\\
&\quad
-\Big(\partial_{x} \big(B(\partial_{x} v_{N}  -\partial_{y} u_{N})+\partial_{z}\theta_{N}
\big)+\partial_{x}\theta_{N}\partial_{z}\Big)\partial_{y}p_{e}^{(1)}\notag\\
&\quad +BL_{3}+\partial_{z}L_{4}=0 ,\label{pe1}
\end{align}
where
\begin{align*} 
\begin{split}
L_{3}&=(\partial_{t}+z \partial_{x} +u_{N}\partial_{x}+v_{N}\partial_{y})(\partial_{x}L_1+\partial_{y}L_2)\\ &\quad
+(-L_2\partial_{x}+L_1\partial_{y})(\partial_{x} v_{N}  -\partial_{y} u_{N}  )\\
    &\quad+(-\partial_{x}L_2+\partial_{y}L_1)(\partial_{x} v_{N}  -\partial_{y} u_{N} )\\ 
    &\quad+\partial_{x}(w_{N}\partial_{z}v_{N})
-\partial_{y}(w_{N}\partial_{z}u_{N}+w_{N})  ,\\
L_{4}&=(-L_2\partial_{x}+L_1\partial_{y}) \theta _{N}+w _{N} \partial_{z} \theta _{N}
-L_1.
\end{split}
\end{align*}
Now we have  a {\it linear equation} of $p_e^{(1)}$ governed by  \eqref{pe1}  satisfying  the boundary conditions $w_e^{(1)}|_{z=0,1}=0:$ 
\begin{align*}
&\partial_{t} \theta_{e}^{(1)}
+z\partial_{x} \theta_{e}^{(1)}
+ u _{N} \partial_{x} \theta_{e}^{(1)} +u_{e}^{(1)} \partial_{x} \theta _{N} 
+ v _{N} \partial_{y}\theta_{e}^{(1)}\\
&\quad +v_{e}^{(1)}\partial_{y} \theta  _{N} 
-v_{e}^{(1)}=0, \ \ \ {\rm at}\ \ \ z=0,1. 
\end{align*}
 Then the  existence  of $p_e^{(1)}$ and then $((u_e^{(1)}, v_e^{(1)}, w_e^{(1)}), \theta_e^{(1)})$  follows from a standard local well-posedness argument. Similarly, $((u_e^{(i)}, v_e^{(i)}, w_e^{(i)}), \theta_e^{(i)}, p_e^{(i)}),$ $ i\geq 1,$ can be identified order by order.   Therefore, we can deduce \eqref{u_prime}-\eqref{div} from \eqref{line1}-\eqref{line5}. The proof of Lemma \ref{lem4.1} is completed.  $\square$

\subsection{Estimates for  the error terms}\label{subsec3.2}
 To begin  with, we state a useful pressure estimate  in the following 
 proposition, which can be deduced from   \cite{TR} by R. Temam.
 \begin{Proposition}[cf. \cite{TR}]\label{prop}
Let $m>5/2$  and  $\textbf{f}\in  H^m(\mathbb{T}^2\times [0,1])$ be a vector function. If $P(x, y, z)$ and $\textbf{u}=(u_1, u_2, u_3)(x, y, z)$   satisfy that
\begin{align*}
&\Delta_{\mathcal{H}} P= {\rm div}  \textbf{f}-{\rm div} (\textbf{u}\cdot\nabla \textbf{u})\ \ \ {\rm in}\ \ \ \ \mathbb{R}^2\times [0,1] \ {\rm or}\  \mathbb{T}^2\times [0,1],\\
&{\rm div} \textbf{u}=0,\\
&\frac{\p P}{\p z}= \textbf{f}\cdot \textbf{n}\ \ \ \ \ \ \ \ \ \ \ \ \ \ \ \ \ \ \ \  \ \ \ {\rm at}\ \ \ \  z=0, 1,
\end{align*}
where $\Delta_\mathcal{H}:=\mathcal{H}^2\p_{xx}^2+\mathcal{H}^2\p_{yy}^2+\p_{zz}$ with positive constant $\mathcal{H},$    
$\textbf{n}=(0, 0, -1)$ at $z=0$ and $\textbf{n}=(0, 0, 1)$ at $z=1,$ respectively.
Then it holds 
\begin{align*}
\|\nabla P \|_{H^m}\leq C \left(\|\textbf{f} \|_{H^m}+\|\textbf{u}\|_{H^m}^2 \right),
\end{align*}
where $C$ is a positive constant.
 \end{Proposition}

Throughout the rest parts, we will use $C_a$ to denote a generic constant depending on 
$(u_a, v_a, w_a)$, $p_a$ and $\theta_a.$

\begin{Lemma}\label{lem4.2}
 For $t>0,$ we have
\begin{align}
&\Vert \sqrt{B}  H(u^{\prime}, v',   Ro w' )(\cdot, t)\Vert^{2}+\|\theta'(\cdot, t)\|^2
\notag\\
&\quad \leq C_a \int_0^t \mathcal{P}\left(\|(\sqrt{B}H u', \sqrt{B} H v', \sqrt{B} Ro w', \theta')(\cdot, t)\|^2\right)dt +\|\theta'(\cdot, 0)\|^2\notag\\
&\quad \quad+ \Vert  \sqrt{B} (u^{\prime},  v',   Ro w')(\cdot, 0)\Vert^{2}+C_a Ro^6 ,\label{lem828}
\end{align}
where $\mathcal{P}(\cdot)$ is a polynomial function.  Hence it holds that 
\begin{align}
\Vert (u^{\prime}, v',    Ro w', \theta' )(\cdot, t)\Vert^{2}\leq C_a(B, H) Ro^6, \ \  \ 0<t<T'
\end{align}
for a time $T' \leq T_a.$ Here $C_a(B,H)$ represents a  generic constant  depending on $B$ and $H,$ additionally. 
\end{Lemma}
{\it Proof.}
 Without loss of generality, we set $H=1.$
Multiplying  (\ref{u_prime})-(\ref{theta_prime}) by $Bu^{\prime}$, $Bv^{\prime}$, $B Ro w^{\prime}$ and $\theta^{\prime}$, respectively,   and then we integrate the result with respect to space variable to get
\begin{align*} 
      &\frac{1}{2}\frac{d}{dt}\Big(B \Vert u^{\prime}\Vert^{2}
+B \Vert v^{\prime}\Vert^{2}
+B \Vert Row^{\prime}\Vert^{2}
+\Vert\theta^{\prime}\Vert^{2}\Big)\\
      &\quad =\frac{1}{Ro}\int \nabla p\cdot (Bu', Bv', BRo w')dxdxdz+K_{1}+K_{2}+K_{3}+ K_{4},
\end{align*}
where
\begin{align*}
K_{1}&=
-B\int\big( u^{\prime}\partial_{x}u_{a}u^{\prime}
+ u_{a}\partial_{x}u^{\prime}u^{\prime}
+ v^{\prime}\partial_{y}u_{a}u^{\prime} 
+ v_{a}\partial_{y}u^{\prime}u^{\prime}\\
      &\quad\quad\quad\quad
+Ro  w ^{\prime}\partial_{z}u_{a}u^{\prime} 
+Ro w_{a}\partial_{z}u ^{\prime}u^{\prime}
+Ro w^{\prime}u^{\prime}
-Ro^3 M_{1}u^{\prime}\big)dxdydz,\\
K_{2}&=
- B\int  \big( u^{\prime}\partial_{x}v_{a}v^{\prime}
+  u_{a}\partial_{x}v^{\prime}v^{\prime} 
+ v^{\prime}\partial_{y}v_{a}v^{\prime} 
+  v_{a}\partial_{y}v^{\prime}v^{\prime}\\
      &\quad \quad \quad\quad
+Ro   w ^{\prime}\partial_{z}v_{a}v^{\prime}
+Ro   w_{a}\partial_{z}v ^{\prime}v^{\prime}
-Ro^3  M_{2}v^{\prime})dxdydz,\\
K_{3}&=
-BRo\int \big( u^{\prime}\partial_{x}w_{a}Row^{\prime}
+Ro u_{a}\partial_{x}w^{\prime}Row^{\prime} +Ro v^{\prime}\partial_{y}w_{a}Row^{\prime}
\\
      &\quad\quad\quad\quad
+Ro  v_{a}\partial_{y}w^{\prime}Row^{\prime} +Ro^2  w^{\prime}\partial_{z}w_{a}Row^{\prime}
+Ro^2  w_{a}\partial_{z}w^{\prime}Row^{\prime} 
\\
 &\quad\quad\quad\quad
-Ro^3 M_{3}w^{\prime}\big)dxdydz,\\
K_{4}&=
-\int\big(  u^{\prime}\partial_{x}\theta_{a}\theta^{\prime} 
+   u_{a}\partial_{x}\theta^{\prime}\theta^{\prime} 
+  v^{\prime}\partial_{y}\theta_{a}\theta^{\prime}  
  +  v_{a}\partial_{y}\theta^{\prime}\theta^{\prime}
 \\
      &\quad\quad\quad\quad
+Ro   w ^{\prime}\partial_{z}\theta_{a}\theta^{\prime} 
+Ro   w_{a}\partial_{z}\theta^{\prime}\theta^{\prime} 
+Ro   v^{\prime}\theta^{\prime} -Ro^3   M_{4}\theta^{\prime}dxdydz.
\end{align*}
 Together with  integration by parts, H\"{o}lder's inequality and Young's inequality, we have
\begin{align*}
\frac{1}{B}|K_{1}| 
        &\leq \|u^{\prime}\| \|\partial_{x}u_{a}\|_{L^\infty}\|u^{\prime}\| 
+\|\partial_{x}u_{a}\|_{L^\infty}\|u^{\prime}\| ^2\\
        &\quad
+\|v^{\prime}\| \|\partial_{y}u_{a}\|_{L^\infty}\|u^{\prime}\| 
+\|\partial_{y}v_{a}\|_{L^\infty}\|u^{\prime}\| ^2\\
        &\quad +\|Row^{\prime}\| \|\partial_{z}u_{a}\|_{L^\infty}\|u^{\prime}\| 
+\|Ro\partial_{z}w_{a}\|_{L^\infty}\|u^{\prime}\| ^2\\
        &\quad +\|Row^{\prime}\| \|u^{\prime}\| 
+Ro^3\|M_{1}\| \|u^{\prime}\| \\
        &\leq C_a \left( \|u^{\prime}\| ^2 +\|v^{\prime}\| ^2 
+\|Row^{\prime}\| ^2+Ro^6\right).
\end{align*}
Similarly, it follows that 
\begin{equation*}
\frac{1}{B}|K_{2}| 
        \leq C_a( \|u^{\prime}\| ^2 +\|v^{\prime}\| ^2 
+\|Row^{\prime}\| ^2+Ro^6),
\end{equation*}
\begin{equation*}
\frac{1}{B}|K_{3}| 
        \leq C_a B  ( \|u^{\prime}\| ^2 +\|v^{\prime}\| ^2 
+\|Row^{\prime}\| ^2+Ro^6),
\end{equation*}
\begin{equation*}
|K_{4}| 
        \leq C_a ( \|u^{\prime}\| ^2 +\|v^{\prime}\| ^2 
+\|Row^{\prime}\| ^2+\|\theta^{\prime}\| ^2+Ro^6).
\end{equation*}
As for the term $\frac{1}{Ro}\int \nabla p\cdot (Bu', Bv', B Ro w') dxdxdz$, with the aid of \eqref{div}, $w'=0$ and $w_a=0$ at $z=0, 1$,
we have 
\begin{align}\label{pl2}
\frac{B}{Ro}\int \nabla p\cdot (u', v', Ro w') dxdxdz=0.
\end{align}

Then we have
\begin{align*} 
      &\frac{1}{2}\frac{d}{dt}\Big(\Vert\sqrt{B} u^{\prime}\Vert_{L^{2}}^{2}
+\Vert\sqrt{B} v^{\prime}\Vert_{L^{2}}^{2}
+ \Vert\sqrt{B} Row^{\prime}\Vert_{L^{2}}^{2}
+\Vert\theta^{\prime}\Vert_{L^{2}}^{2}\Big)\\
      &\quad
 \leq  C_a \mathcal{P}( \|(\sqrt{B}u', \sqrt{B}v', \sqrt{B}Ro w', \theta')\|_{L^2}^2)+C_a Ro^6.
\end{align*} 
The proof of this lemma is completed.  
$\square$

As corollary   of  Lemma \ref{lem4.2}, we directly have the following results.
\begin{Lemma}\label{lem4.3}
The solution $(u', v', w', \theta')$ of error equations \eqref{u_prime}-\eqref{div} defined on $ [0, T']$  satisfy
\begin{align*}
&\|(u', v',  w', \theta')\|_{L^2}^2\leq C(B,H) Ro^4.
\end{align*}
Let $(\textbf{u}, \theta)$ denote the solution of  \eqref{original} \eqref{orbc}.  Therefore,  for $t\in [0, T'],$  it holds that 
\begin{align}
&\|(\textbf{u}-U-\textbf{u}_N,\ \theta-\Theta-\theta_N)\|_{L^2}^2\leq C(B,H) Ro^2,\label{l2ro}
\end{align}
where $\textbf{u}_N=(u_N, v_N, w_N).$ 
\end{Lemma} 

Now we are ready to prove our main theorem on the nonlinear stability of shear flow for any small  fixed $Ro>0.$

{\it Proof of Theorem \ref{mainth}.}
We denote the initial perturbation of the shear flow as $(\textbf{u}^*_0, \theta^*_0).$
Let  $\|(\textbf{u}^*_0,  \theta^*_0)\|_{L^2}= \delta$   for  example.   It follows from \eqref{l2ro} that 
$$\|(\textbf{u}^*-\textbf{u}_N,\theta^*-\theta_N)(t)\| \leq \sqrt{C(B,H)} Ro,$$
which implies 
\begin{align}
\|(\textbf{u}_N, \theta_N)\|-  \sqrt{C(B,H)} Ro\leq \|(\textbf{u}^*, \theta^*)\|\leq \|(\textbf{u}_N, \theta_N)\|+\sqrt{C(B,H)} Ro.
\end{align}

Without loss of generality, we assume  $\|(\textbf{u}_N, \theta_N)(t=0)\|_{L^2}=\|(\textbf{u}^*_0, \theta^*_0)\|_{L^2}.$  If  the instability condition $q<q_c$ holds,  it follows from Lemma 3.2  that   the solution $(\textbf{u}_N, \theta_N)$ generates   a growing mode $e^{\alpha \Im (c) t}$ for all $t\geq 0.$   Then there exists a time 
$$T^\delta :=\frac{1}{\alpha \Im(c)}\ln \frac{2\sqrt{C(B, H)}Ro+\eta}{\delta},\ \ \ {\rm for\ \ some}\ \ 0<\delta<\eta \leq 1,$$
such that 
\begin{align*}
\|(\textbf{u}^*,\theta^*)(T^\delta)\|\geq \|(\textbf{u}_N, \theta_N)(T^\delta)\|-\sqrt{C(B,H)} Ro>\eta.
\end{align*}

Next, we denote    the maximal  existence time of solutions  $(\mathbf{u}^*,\theta^*)$ as  $T.$   It follows from Proposition \ref{prop}  and  blow-up criteria  for  $(\textbf{u}^*,\theta^*)$ with fixed $Ro>0$ (cf. \cite{BKM}) that:  for any 
$t\in[0, T],$
$$ \|(\mathbf{u}^*,\theta^*)(\cdot,t)\|_{H^3}\leq C\|(\mathbf{u}^*,\theta^*)(\cdot, 0)\|_{H^3}\big( e^{\int_0^T [\|(\nabla \textbf{u}^*, \nabla \theta^*)(\tau)\|_{L^\infty}(1+1/Ro)]d\tau}\big)< +\infty.$$
Indeed, for $T^\delta$ it holds that 
\begin{align}
&\|(\textbf{u}^*,\theta^*)(T^\delta)\|_{H^3}\leq C\|(\mathbf{u}^*,\theta^*)(\cdot, 0)\|_{H^3} \big(e^{(3C \sqrt{C(B, H)}Ro+\eta)T^\delta(1+1/Ro)}\big)<+\infty,\notag
\end{align}
for  a fixed small $Ro>0.$
Hence we have $T^\delta\leq T.$ The proof of Theorem \ref{mainth} is completed. \hfill $ \square$

\section{Appendix}
In this section, we first  state the physical model of  inviscid non-conducting Boussinesq system in a
frame of reference rotating with steady angular velocity $\Omega \mathbf{k}$ about
the vertical direction. Then the  derivation of system \eqref{original} will be given.
The original system reads
\begin{equation}\label{original*}
\begin{cases}
\partial_{t_*} \mathbf{u}_* +\mathbf{u}_* \cdot \nabla_* \mathbf{u}_*+2 \Omega \mathbf{k} \times \mathbf{u}_*  =-\rho_0^{-1} \nabla_* p_*+\gamma g(\theta_*-\theta_0) \mathbf{k},\\
\partial_{t_*} \theta_* +\mathbf{u}_* \cdot \nabla_* \theta_*  =0,\\
\boldsymbol{\nabla}_* \cdot \mathbf{u}_* =0,
\end{cases}
\end{equation}
where $\mathbf{u}_*=(u_*,v_*,w_*)$ is the velocity relative to the rotating frame. $\theta_0$ and $\rho_0$ are, respectively,   the reference temperature and density. Here $\gamma, g$ are positive physical  constants and   $\rho_*=\rho_0(1-\gamma(\theta_*-\theta_0))$, $p_*$   and $\theta_*$  denote the density,  the relative pressure and the temperature of the fluid, respectively.  Here $\Omega$ and $\mathbf{k}$ represent, respectively, the reference  rotation rate and  the  unit vector in the $z_*$-direction, i.e., $\mathbf{k}=(0, 0, 1).$

In this study,  the space variable $(x_*, y_*, z_*)\in \mathbb{T}^2\times[0, H]$  and  $x_*, y_*$ are period of length $L$.   In the $z_*$ direction we assume  a  non-penetration  boundary condition: 
\begin{align*}
w_*=0 \quad \text { at } z_*=0,\ H. \label{bc}
\end{align*}
Suppose that the basic state is given by
\begin{equation} \label{basic*}
\begin{cases}
\mathbf{U}_*=V H^{-1} (z_*, 0, 0),\\
\Theta_*=\theta_0+\bar{\theta} H^{-1} z_*-2 \Omega V(\gamma g H)^{-1} y_*,\\
P_*=\rho_0\big(\frac{1}{2} \gamma g\bar \theta H^{-1} z_*^2-2 \Omega V H^{-1} y_* z_*\big),
\end{cases}
\end{equation} 
where 
\[ (x_*,  y_*)\in \mathbb{T}^2, \ \  0 \leq z_* \leq H, \]
and $V$ is a reference speed of the basic state and $\bar \theta$ is a constant scale of basic vertical temperature difference. 
The distribution of $\Theta_*$  indicates the distribution of density $\rho_*$. The instability of the basic state \eqref{basic*} is called baroclinic instability in meteorology.   It is called {\it baroclinic instability} because it depends essentially upon the difference between the surfaces of constant density
and of constant pressure in the fluid ( see Chapter 6 in \cite{Drazin} for the physical mechanism).  

 Then it is convenient to define the dimensionless variables
\begin{equation*}
\begin{split}
(x, y) & =(x_*, y_*)/L, \quad z=z_* / H, \quad t=V t_*/L, \\
(u, v) & =(u_*, v_*) / V, \quad w=Lw_* / V H R o, \\
\theta & =\gamma g H\big(\theta_*-\theta_0-\bar \theta z_* / H\big) / 2 \Omega VL, \\
p & =\big(p_*-\frac{1}{2} \gamma \rho_0 g\bar \theta H^{-1} z_*^2\big) / 2 \Omega V L\rho_0 ,
\end{split}
\end{equation*}
where the Rossby number, a characteristic ratio of the inertia to the Coriolis force, is given by
$$
R o=V / 2 \Omega L.
$$
We will use $B$ to denote Burger number:
$$B=\gamma g H \bar\theta/ (4\Omega^2 L^2).$$
Now the basic state \eqref{basic*} reduces to 
\begin{align}\label{shear*}
\mathbf{U}=(z, 0, 0), \quad \Theta=-y, \quad P=-y z,
\end{align}
where $0\leq z\leq 1.$
Finally, the system \eqref{original*} can be rewritten   as \eqref{original}.

\vspace{0.5cm}

\noindent {\bf {Acknowledgments.}} The authors would like to thank professor  Qingtang Su at Academy of Mathematics and Systems  Science, CAS  for his  guidance  and  insightful   discussion. 
Mao's research  was supported in part by  Postdoctoral Fellowship Program of CPSF under grant GZC20232910 and China Postdoctoral
Science Foundation under grant   2024M753430.  Wang's research  was   supported in part by Natural Science Foundation of China (NSFC)   under grants 12426630, 12426632, 12101350, 12271284;  Zhejiang Provincial Natural Science Foundation of China under grant LMS26A010013; and  Scientific Research Fund of Zhejiang  Provincial Education Department under grant Y202454255.

\vspace{0.5cm}

\noindent\textbf{Data availability.}
Data sharing not applicable to this article as no datasets were generated or analysed during the current study.

\noindent\textbf{Conflict of interest.}
The authors declare that they have no conflict of interest.

\end{document}